\documentclass[
final
]{dmtcs-episciences}

\usepackage[utf8]{inputenc}
\usepackage{subfigure}

\usepackage[round]{natbib}

\newcommand{\HH}{\mathcal{H}}
\newcommand{\EE}{\mathcal{E}}
\newcommand{\GG}{\mathcal{G}}
\newtheorem{theorem}{Theorem}
\newtheorem{corollary}[theorem]{Corollary}
\newtheorem{lemma}[theorem]{Lemma}

\newtheorem{observation}[theorem]{Observation}

\newtheorem{problem}[theorem]{Problem}

\author[Gábor Damásdi et al.]{Gábor Damásdi\affiliationmark{1,2}\thanks{Partially supported by ERC Advanced Grant GeoScape.}
  \and Balázs Keszegh\affiliationmark{1,2}\thanks{Research supported by the ERC Advanced Grant no.~101054936 ERMiD, the J\'anos Bolyai Research Scholarship of the Hungarian Academy of Sciences, by the National Research, Development and Innovation Office -- NKFIH under the grant K 132696 and FK 132060, by the \'UNKP-23-5 New National Excellence Program of the Ministry for Innovation and Technology from the source of the National Research, Development and Innovation Fund and by the Thematic Excellence Program TKP2021-NKTA-62 of the National Research, Development and Innovation Office.}\\
  \and Dömötör Pálvölgyi\affiliationmark{2,1}\thanks{Partially supported by the NRDI EXCELLENCE-24 grant no.~151504 Combinatorics and Geometry and by the ERC Advanced Grant no.~101054936 ERMiD, and earlier by the J\'anos Bolyai Research Scholarship of the Hungarian Academy of Sciences, and by the New National Excellence Program \'UNKP-23-5 and by the Thematic Excellence Program TKP2021-NKTA-62 of the National Research, Development and Innovation Office.}
  \and Karamjeet Singh\affiliationmark{3}\thanks{Supported by Overseas Research Fellowship of IIIT-Delhi, India.}}

\title{The complexity of recognizing $ABAB$-free hypergraphs}
\affiliation{
  HUN-REN Alfréd Rényi Institute of Mathematics, Budapest, Hungary.\\
  ELTE Eötvös Loránd University, Budapest, Hungary.\\
  IIIT-D Indraprastha Institute of Information Technology, Delhi, India.}

\begin{document}

\publicationdata{vol. 27:1, Permutation Patterns 2024}{2025}{8}{10.46298/dmtcs.14610}{2024-10-22; 2024-10-22; 2025-03-31}{2025-04-02}

\maketitle
\begin{abstract}
The study of geometric hypergraphs gave rise to the notion of $ABAB$-free hypergraphs. A hypergraph $\mathcal{H}$ is called $ABAB$-free if there is an ordering of its vertices such that there are no hyperedges $A,B$ and vertices $v_1,v_2,v_3,v_4$ in this order satisfying $v_1,v_3\in A\setminus B$ and $v_2,v_4\in B\setminus A$. In this paper, we prove that it is NP-complete to decide if a hypergraph is $ABAB$-free. We show a number of analogous results for hypergraphs with similar forbidden patterns, such as $ABABA$-free hypergraphs. As an application, we show that deciding whether a hypergraph is realizable as the incidence hypergraph of points and pseudodisks is also NP-complete.
\end{abstract}

\section{Introduction}

Let $\HH=(V,\EE)$ be a hypergraph with an ordered vertex set and $k\in\mathbb{N}$. We say that edges $H$ and $L$ \emph{form an $(AB)^k$ pattern in $\HH$} if there exist $h_1,l_1,h_2,l_2,\ldots, h_k,l_k \in V$ in this order such that $h_i\in H\setminus L$ and $l_j\in L\setminus H$ for all $1\le i\le k$ and $1\le j\le k$.\footnote{Equivalently, $H$ and $L$ form an $(AB)^k$ pattern if writing down respectively the symbol $h$ or $l$ for each vertex contained in only one of the hyperedges, we do \emph{not} get a Davenport–Schinzel sequence of order $2k-2$.}  An ordering $\pi$ of $V$ is said to be \emph{$(AB)^k$-free}, if no pair of edges form an $(AB)^{k}$ pattern. A hypergraph $\HH=(V,\EE)$ is called \emph{$(AB)^k$-free} if there is an ordering $\pi$ of $V$ with which $\HH$ is $(AB)^k$-free.

Similarly, $H$ and $L$ \emph{form an $(AB)^{k-1}A$ pattern in $\HH$} if there exist $h_1,l_1,h_2,l_2,\ldots h_{k-1},l_{k-1},h_k \in V$ in this order such that $h_i\in H\setminus L$ and $l_j\in L\setminus H$ for all $1\le i\le k$ and $1\le j\le k-1$, and we define \emph{$(AB)^{k-1}A$-free} orderings and \emph{$(AB)^{k-1}A$-free} hypergraphs accordingly.

$ABA$-free hypergraphs were introduced by the two middle authors \cite{abafree} to study pseudohalfplane arrangements from an abstract viewpoint, and this enabled them to generalize a result of Smorodinsky and Yuditsky about polychromatic colorings for halfplanes \cite{sy12}.
The notion was also generalized to a higher number of alternations in the same paper; for more background, see \cite{ackerman2020coloring}, or for a recent application, see \cite{lenovo}.
Similar alternations of hypergraphs were also studied independently in an entirely different context to bound the chromatic number of certain generalized Kneser graphs  \cite{hajiabol15}.

In this paper, we study the problem of deciding whether an input hypergraph is $(AB)^k$-free or not.
This is equivalent to whether it can be realized as a specific geometric hypergraph.
More precisely, a hypergraph $\HH=(V,\EE)$ is $(AB)^k$-free exactly if there are maps $\alpha\colon V\to \mathbb R^2$ and $\gamma\colon \EE\to \mathbf C(\mathbb R)$ (the set of continuous $\mathbb R\to \mathbb R$ functions) such that $\alpha(v)$ is over the graph of the function $\gamma(E)$ if and only if $v\in E$, and for any two distinct $E,E'\in\EE$, the graphs of $\gamma(E)$ and $\gamma(E')$ intersect at most $2k-2$ times.
A similar statement is true for $(AB)^kA$-free hypergraphs and $2k-1$ intersections.
If $k=2$, i.e., for $ABAB$-free hypergraphs, there is an equivalent characterization with the interiors of Jordan curves that intersect at most twice and enclose a common point; see \cite{ackerman2020coloring}.

We note that $ABAB$-freeness of 2-uniform hypergraphs can be decided in polynomial time. In \cite{ackerman2020coloring} it was shown that $ABAB$-free graphs are outerplanar. The converse is also true, any outerplanar graph is $ABAB$-free. Indeed, we can simply take the vertices in the order they appear along the outer face. (Cut vertices appear more than once along the boundary, we keep an arbitrary one of these.) Any $ABAB$ pattern would correspond to crossing chords, hence the ordering is $ABAB$-free.

We show the following results.

\begin{theorem}\label{thm:even}
For any positive integer $k\ge 2$, it is NP-complete to decide if a given hypergraph is $(AB)^k$-free.
\end{theorem}

\begin{theorem}\label{thm:odd}
For any positive integer $k\ge 2$, it is NP-complete to decide if a given hypergraph is $(AB)^{k}A$-free.
\end{theorem}

Let $\pi$ be any ordering of the vertices of a hypergraph $\mathcal{H}$, and $k$ be any fixed positive integer.
It is easy to see that one can decide in polynomial time if $\mathcal{H}$ is $(AB)^k$-free or if it is $(AB)^kA$-free w.r.t. the ordering $\pi$.
This shows that the decision problems in Theorem \ref{thm:even} and Theorem \ref{thm:odd} are in class NP.
Therefore, to prove the theorems, it is enough to show the NP-hardness of these problems.

In Sections \ref{sec:abab} and \ref{sec:ababa} we show that deciding $ABAB$-freeness and deciding $ABABA$-freeness are NP-hard. Then in Section \ref{sec:oddeven} we show how Theorem \ref{thm:even} and Theorem \ref{thm:odd} follow from these results.
As a geometric application, we show in Section \ref{sec:questions} that it is NP-complete to decide whether a given hypergraph is realizable as the incidence hypergraph of points and a family of pseudodisks.
These results leave the complexity of $ABA$-freeness undecided; this is discussed in Section \ref{sec:open}.

\section{$ABAB$-freeness is NP-complete}\label{sec:abab}
In this section we show the following special case of Theorem \ref{thm:even}.

\begin{theorem}\label{thm:abab}
It is NP-hard to decide if a given hypergraph is $ABAB$-free.
\end{theorem}

For the proof of Theorem \ref{thm:abab} we need a couple of notations and lemmas. Throughout the proof we will have a fixed $N$, and the vertex set of our hypergraph will be $\{1,2,\ldots, N\}$. So for the following definitions and lemmas assume that $N$ is fixed. For $S\subseteq\{1,2,\ldots, N\}$ we say that $S$ is an \emph{interval}, if the elements of $S$ are consecutive, and $S$ is a \emph{$t$-interval} if it is the union of $t$ intervals (some of which might be empty).
For $x<y$, let $[x,y]$ denote the interval $\{x,\dots, y\}$.

Furthermore, $S$ is a \emph{circular interval} if it is either an interval or a 2-interval containing both $1$ and $N$. $S$ is a \emph{circular $t$-interval} if it is the union of $t$ circular intervals.
For $x<y$, let $[y,x]$ denote the circular interval $[y,N]\cup [1,x]$.

For a permutation $\pi$ we say that $\pi'$ is a \emph{cyclic shift of $\pi$} if we can get $\pi'$ from $\pi$ by repeatedly moving the last entry to the first position. We say that two orderings are \emph{equivalent} if we can get one from the other by applying cyclic shifts, or by reversing the order and then applying cyclic shifts. 

\begin{observation}\label{obs:cyclic}
    If $\pi_1$ and $\pi_2$ are equivalent, then $\pi_1$ is an $(AB)^k$-free ordering if and only if $\pi_2$ is $(AB)^k$-free ordering.
\end{observation}

Note that Observation \ref{obs:cyclic} does not hold for $(AB)^kA$-free hypergraphs.

For a set $S\subset \{1,2,\dots,N\}$ let $S^c$ denote the complement of $S$, that is $S^c=\{1,2,\dots,N\}\setminus S$. Note that the complement of a circular interval is also a circular interval.

\begin{lemma}\label{lem:complement}
Let $V=\{1,2,\ldots, N\}$ and $\EE$ be any collection of subsets of $V$. Assume that $\EE$ contains $S$ and also $S^c$ for some $S\subset V$.
\begin{enumerate}
\item[a)] If $\pi$ is an $ABAB$-free ordering of $\HH=(V,\EE)$, then either the elements of $S$ or the elements of $S^c$ are consecutive in $\pi$.
(That is, $S$ and $S^c$ are circular intervals with respect to the ordering $\pi$.)
\item[b)] If the vertices of $S$ or the elements of $S^c$ are consecutive in $\pi$, then neither $S$ nor $S^c$ forms an $ABAB$ pattern with any subset of $V$. 
\end{enumerate}

\end{lemma}
\begin{proof}
For part a) note that if neither the elements of $S$ nor the elements of $S^c$ are consecutive, then $S$ and $S^c$ together form an $ABAB$ pattern.

For part b) we can use Observation \ref{obs:cyclic} and move $S$ (or $S^c$) to the beginning of the ordering. If we have an interval at the beginning of the order, it clearly cannot form an $ABAB$ pattern with any other subset of $V$.   
\end{proof}

\begin{lemma}\label{lem: intervals}
Let $V=\{1,2,\ldots, N\}$ where $N\ge 4$ and let $\EE$ be the set of all circular intervals of $[N]$.
Then the hypergraph $\HH=(V,\EE)$ with an ordering $\pi$ is $ABAB$-free if and only if $\pi$ is equivalent to the increasing order $1,2,\ldots, N$. 
\end{lemma}
\begin{proof}
First, consider the order $1,2,\dots, N$  and a circular interval $S$.
If $x,y\in S$ and $x<y$, then either $[x,y]\subset S$ or $[1,x]\cup[y,N]\subset S$, hence $S$ is not part of any $ABAB$ pattern. As this holds for any $S$, the ordering is $ABAB$-free. 

In the other direction, let $\pi$ be an ordering of $V$ not equivalent to $1,2,\ldots,N$. Since $N\ge 4$, we can find $x,y,z,q\in V$ such that $x<y<z<q$, but in $\pi$ their cyclic order is not $x,y,z,q$ nor $q,z,y,x$. 

\begin{itemize}
    \item If the cyclic order in $\pi$ is $x,y,q,z$ or $x,z,q,y$, then $[q,x]$ and $[y,z]$ are disjoint circular intervals forming an $ABAB$ pattern.
    \item If the cyclic order in $\pi$ is $x,z,y,q$ or $x,q,y,z$, then $[x,y]$ and $[z,q]$ are disjoint circular intervals forming an $ABAB$ pattern.
\end{itemize}

Hence, any $ABAB$-free ordering is equivalent to $1,2,\ldots, N$.
    \end{proof}

Let $V=\cup_{i=1}^kA_i$ for some disjoint $A_i\subset V$.
We say that an ordering of $V$ \emph{has structure $A_1,\dots,A_k$} if for any $i\neq j$ and $v\in A_i, w\in A_j$, the element  $v$ comes before $w$ if and only if $i<j$. That is, other than the internal order within the $A_i$-s, the ordering is fixed. 

\begin{observation}\label{obs:structure}
 An ordering $\pi$ has structure $A_1,\dots,A_k$ if and only if for any selection $a_i\in A_i$ the $a_i$-s are in order within $\pi$.
    
\end{observation}

Let $\mathcal{I}_k$ denote the collection of all circular intervals on $\{1,2,\ldots, k\}$.
From Lemma \ref{lem: intervals} and Observation \ref{obs:structure} we obtain the following.

\begin{corollary}\label{cor:circularblocks}
Let $V=\cup_{i=1}^kA_i$ for some disjoint $A_i\subset V$ and let $\EE=\{\bigcup\limits_{i\in\alpha}A_i:\alpha\in\mathcal{I}_k\}$.
Then an ordering $\pi$ of $V$ is $ABAB$-free if and only if $\pi$ is equivalent to an ordering that has structure $A_1,\ldots, A_k$.
\end{corollary}

Now we are ready to prove Theorem \ref{thm:abab}.

\begin{proof}
Deciding whether a 3-uniform hypergraph has a proper 2-coloring is NP-hard \cite{inproceedings}. We reduce this problem to deciding  $ABAB$-freeness, thus showing the theorem. 

Let $\GG=(V_{\GG},\EE_{\GG})$ be any given 3-uniform hypergraph where $V_{\GG}=\{v_1,v_2,\ldots, v_n\}$ and $\EE_{\GG}=\{e_1,e_2,\ldots, e_m\}$.
We construct a hypergraph $\HH=(V,\EE)$ such that $\GG$ is 2-colorable if and only if $\HH$ is $ABAB$-free.

The vertex set $V$ of $\HH$ is the set of first $N$ natural numbers constructed as follows: for each vertex of $\GG$ we take two vertices, and for each edge of $\GG$ we take four vertices.
That is, let $V_1=\bigcup_{i=1}^n R_i$ where $R_i=\{2i-1,2i\}$ and let $V_2=\bigcup_{j=1}^m S_j$, where $S_j=\{2n+(4j-3),2n+(4j-2),2n+(4j-1),2n+4j\}$. The vertex set of $\HH$ is $V=V_1\cup V_2$. We will denote by $t_j$ the first vertex of $S_j$, that is, $t_j=2n+(4j-3)$ and $S_j=\{t_j,t_j+1,t_j+2,t_j+3\}$.

Now, we will construct the sets of hyperedges in three steps. $\EE_1$ is the set of circular intervals of $[N]$ that can be written as the unions of $R_i$-s and $S_j$-s.
Corollary \ref{cor:circularblocks} implies that any $ABAB$-free ordering of $(V,\EE_1)$ is equivalent to a unique $ABAB$-free ordering that has structure $R_1,\dots,R_n,S_1,\dots,S_m$. 

Our next goal is to construct a set of edges $\EE_2$, such that any $ABAB$-free ordering of $(V,\EE_1\cup \EE_2)$ corresponds to a 2-coloring of the vertex set $V_\GG$,
in such a way, that the color of a vertex $v_i$ can be determined from the order within $R_i$, and furthermore, if $v_i\in e_j$, then the color can also be determined from the order within $S_j$. 
Take an $ABAB$-free ordering $\pi_1$ of $(V,\EE_1)$, and let $\pi_2$ be the unique $ABAB$-free ordering that is equivalent to $\pi_1$ and has structure $R_1,\dots,R_n,S_1,\dots,S_m$.
We will say that the color of $v_i$ is decided based on the order of $2i-1$ and $2i$ in $\pi_2$.
We color $v_i$ red if $2i-1$ precedes $2i$ in $\pi_2$, and color it blue otherwise. 

Consider an edge $e_j=\{v_i,v_k,v_l\}$ with $i<k<l$ and the corresponding set $S_j=\{t_j,t_j+1,t_j+2,t_j+3\}$.
Our goal is to ensure that the order of $t_j$ and $t_j+1$ in $\pi_2$ represents the color of $v_l$, the order of $t_j+1$ and $t_j+2$ represents the color of $v_k$, and the order of $t_j+2$ and $t_j+3$ represents the color of $v_i$.
We emphasize that $i<k<l$, that is, the pairs $(t_j,t_j+1)$, $(t_j+1,t_j+2)$, $(t_j+2,t_j+3)$ represent the color of the vertices of the edge in reverse order. For example, the orders $(t_j,t_j+1,t_j+2,t_j+3)$ and $(t_j+3,t_j+2,t_j+1,t_j)$ correspond to the all red and all blue colorings of the vertices in edge $e_j$, respectively.       

To achieve this, for each $v_i$ and edge $e_j$ where $v_i$ is the $r$-th smallest vertex in $e_j$  we add the edges $I_{i,j,r}=\{2i,\dots,t_j+(3-r)\}$ and $I'_{i,j,r}=\{2i-1,2i+1,2i+2,\dots,t_j+(3-r)-1,t_j+(3-r)+1\}$ to $\EE_2$.
Now $I_{i,j,r}$ and $I'_{i,j,r}$ differ only on 4 elements, $2i-1,2i,t_j+(3-r),t_j+(3-r)+1$, and the structure forced by $\EE_1$ ensures that the first two always come sooner in $\pi_2$ than the second two.
Edges $I_{i,j,r}$ and $I'_{i,j,r}$ enforce that in any $ABAB$-free ordering $2i-1$ precedes $2i$ if and only if $t_j+(3-r)$ precedes $t_j+(3-r)+1$.
See Table \ref{tab:edges} for the six edges corresponding to a fixed $e_j=\{v_i,v_k,v_l\}$.

\begin{table}[!h]
\centering
\begin{tabular}{|l|ll|c|ll|c|ll|c|llll|}
\hline

 & \multicolumn{2}{c|}{$R_i$} & \multicolumn{1}{c|}{$\dots$} & \multicolumn{2}{c|}{$R_k$} & \multicolumn{1}{c|}{$\dots$} & \multicolumn{2}{c|}{$R_l$} & \multicolumn{1}{c|}{$\dots$} & \multicolumn{4}{c|}{$S_j$}                                                   \\ \hline
 $I_{i,j,1}$ & \multicolumn{1}{l|}{}  & \textbf{x} & x                            & \multicolumn{1}{l|}{x} & x & x                            & \multicolumn{1}{l|}{x} & x & x                            & \multicolumn{1}{l|}{x} & \multicolumn{1}{l|}{x} & \multicolumn{1}{l|}{\textbf{x}} &   \\ \hline
 $I'_{i,j,1}$ & \multicolumn{1}{l|}{\textbf{x}} &   & x                            & \multicolumn{1}{l|}{x} & x & x                            & \multicolumn{1}{l|}{x} &  x & x                            & \multicolumn{1}{l|}{x} & \multicolumn{1}{l|}{x} & \multicolumn{1}{l|}{}  & \textbf{x} \\ \hline
 $I_{k,j,2}$ & \multicolumn{1}{l|}{}  &   &                              & \multicolumn{1}{l|}{}  & \textbf{x} & x                            & \multicolumn{1}{l|}{x} & x & x                            & \multicolumn{1}{l|}{x} & \multicolumn{1}{l|}{\textbf{x}} & \multicolumn{1}{l|}{}  &   \\ \hline
$I'_{k,j,2}$ & \multicolumn{1}{l|}{}  &   &                              & \multicolumn{1}{l|}{\textbf{x}} &   & x                            & \multicolumn{1}{l|}{x} & x & x                            & \multicolumn{1}{l|}{x} & \multicolumn{1}{l|}{}  & \multicolumn{1}{l|}{\textbf{x}} &   \\ \hline
 
 $I_{l,j,3}$ & \multicolumn{1}{l|}{}  &   &                              & \multicolumn{1}{l|}{}  &   &                              & \multicolumn{1}{l|}{}  & \textbf{x} & x                            & \multicolumn{1}{l|}{\textbf{x}} & \multicolumn{1}{l|}{}  & \multicolumn{1}{l|}{}  &   \\ \hline
 $I'_{l,j,3}$ & \multicolumn{1}{l|}{}  &   &                              & \multicolumn{1}{l|}{}  &   &                              & \multicolumn{1}{l|}{\textbf{x}} &   & x                            & \multicolumn{1}{l|}{}  & \multicolumn{1}{l|}{\textbf{x}} & \multicolumn{1}{l|}{}  &   \\ \hline
\end{tabular}
\caption{Six edges of $\HH$ in $\EE_2$ corresponding to the edge $e_j=\{v_i,v_k,v_l\}$ of $\GG$.}
\label{tab:edges}
\end{table}

So any $ABAB$-free ordering of $(V,\EE_1\cup \EE_2)$ corresponds to a unique coloring of $\GG$, but the coloring might not be proper. For later use we note the following. 

\begin{observation}\label{obs:almostinterval}
 Let $E\in \EE_1\cup \EE_2$ and let $D_1$ be the first and $D_2$ the last of the sets $R_1,\dots, R_n,S_1,\dots,S_m$ that intersects $E$. Then the sets in the list $R_1,\dots, R_n,S_1,\dots,S_m$ between $D_1$ and $D_2$ are subsets of $E$.    
\end{observation}

Now, we add some edges $\EE_3$ to ensure that each $ABAB$-free ordering corresponds to a proper 2-coloring.
For each $S_j=\{t_j,t_j+1,t_j+2,t_j+3\}$ add $\{t_j,t_j+3\}$ and $\{t_j,t_j+3\}^c$ to $\EE_3$.
By Lemma \ref{lem:complement}, the possible orderings of $S_j$ in an $ABAB$-free ordering of $(V,\EE_1\cup\EE_2\cup\EE_3)$ are the ones where $t_j$ and $t_j+3$ are neighbors.
These do not include the orders that correspond to monochromatic edges, namely the orders $(t_j,t_j+1,t_j+2,t_j+3)$ and $(t_j+3,t_j+2,t_j+1,t_j)$.
Hence, each $ABAB$-free ordering of $(V,\EE_1\cup\EE_2\cup\EE_3)$ corresponds to a proper coloring. 

We are left with one task, we have to show that for each proper 2-coloring there is an $ABAB$-free ordering.
We take the vertices in increasing order, and then within each $R_i$ and $S_i$, we reorder the vertices. This way the resulting ordering $\pi$ will automatically have structure $R_1,\dots,R_n,S_1,\dots,S_m$.
Order an $R_i$ as $2i-1,2i$ if $v_i$ is red and as $2i,2i-1$ if $v_i$ is blue.
Order an $S_j=\{t_j,t_j+1,t_j+2,t_j+3\}$ that corresponds to $e_j=\{v_i,v_k,v_l\}$ with $i<k<l$ based on the colors of $v_i,v_k,v_l$ according to Table \ref{tab:ordering}. 

\begin{table}[!h]
\centering
\begin{tabular}{|l|l|}
\hline
colors  of $v_i,v_k,v_l$  & order in $S_j$\\
\hline
\hline
red, red, blue & $t_j+1, t_j+2, t_j+3, t_j $  \\
\hline
red, blue, red & $t_j+2$, $t_j+3$, $t_j$, $t_j+1$                  \\
\hline
red, blue, blue &  $t_j+2, t_j+1, t_j, t_j+3$,                  \\
\hline
blue, red, red &   $t_j+3, t_j, t_j+1, t_j+2$               \\
\hline
blue, red, blue &   $t_j+1$, $t_j$, $t_j+3$, $t_j+2$                 \\
\hline
blue, blue, red&  $t_j, t_j+3 ,t_j+2, t_j+1, $                \\
\hline
                  
\end{tabular}
\caption{Ordering rule for the $S_j$-s.}
\label{tab:ordering}
\end{table}

We can see in Table \ref{tab:ordering} that in each case the pairs $(t_j,t_j+1)$, $(t_j+1,t_j+2)$ and $(t_j+2,t_j+3)$ represent the colors of $v_l, v_k,$ and $ v_i$ correctly.
Hence this ordering indeed represents the desired coloring. 

Now we will show that $\pi$ is indeed an $ABAB$-free ordering.
Since the ordering has the structure $R_1,\dots,R_n,S_1,\dots,S_m$, part b) of Lemma \ref{lem:complement} tells us that the edges of $\EE_1$ cannot take part in an $ABAB$ pattern.
Similarly, as in each $S_j=\{t_j, t_j+1, t_j+2, t_j+3\}$ we placed $t_j$ adjacent to $t_j+3$, the edges in $\EE_3$ cannot form an $ABAB$ pattern with any edge.

Hence, we only have to check whether two edges from $\EE_2$ form an ABAB pattern.
First consider the case when we have two edges corresponding to the same $e_j=\{v_i,v_k,v_l\}$, that is, two rows from Table \ref{tab:edges}. The edge pairs ($I_{i,j,1}$, $I'_{i,j,1}$), ($I_{k,j,2}$, $I'_{k,j,2}$), ($I_{l,j,3}$, $I'_{l,j,3}$) will not form $ABAB$ patterns since the ordering represents the coloring correctly.
Further, let $H\in\{I_{i,j,1},I'_{i,j,1}\},\;K\in\{I_{k,j,2},I'_{k,j,2}\}$, and $L\in\{I_{l,j,3},I'_{l,j,3}\}$.
Then $H\supset K\supset  L$.
Thus, no two edges from $\{I_{i,j,1}, I'_{i,j,1}, I_{k,j,2}, I'_{k,j,2}, I_{l,j,3}, I'_{l,j,3}\}$ can form an $ABAB$ pattern.

Secondly, suppose $H\in \{I_{i_1,j_1,r_1}, I'_{i_1,j_1,r_1}\}$ and $L\in \{I_{i_2,j_2,r_2}, I'_{i_2,j_2,r_2} \}$ where $j_1<j_2$ and suppose that $H$ and $L$ form an $ABAB$ pattern on vertices $X=(x_1,x_2,x_3,x_4)$.
The vertices between $R_{max(i_1,i_2)}$ and $S_{j_1}$ belong to both $H$ and $L$, hence they cannot be in $X$.
Since $j_1<j_2$, the vertices of $S_{j_1}$ and later vertices cannot be in $H\setminus L$, so we have at most one vertex of $X$ there.
Now consider three cases. If $i_1=i_2$ we have only two remaining vertices that we can put in $X$, the vertices $2i_1-1$ and $2i_1$, but two vertices are not enough.
If $i_1<i_2$, then the vertices in $R_{i_2}$ and before that cannot belong to $L\setminus H$, so we have at most one vertex of $X$ there, which is not enough.
Similarly, if $i_2<i_1$ the vertices in $R_{i_1}$ and before that cannot belong to $H\setminus L$, so we have at most one vertex of $X$ there, which is not enough. 

Therefore the ordering $\pi$ is $ABAB$-free, and the proof is complete. 
\end{proof}

\section{$ABABA$-freeness is NP-complete}\label{sec:ababa}
In this section we show the following special case of Theorem \ref{thm:odd}.

\begin{theorem}\label{thm:ababa}
It is NP-hard to decide if a given hypergraph is $ABABA$-free.
\end{theorem}

We start with a lemma similar to Lemma \ref{lem: intervals}. Note that Observation \ref{obs:cyclic} does not hold in the $ABABA$-free setting, but an ordering is $ABABA$-free if and only if the reverse of the ordering is $ABABA$-free.

\begin{lemma}\label{lem: two intervals}
Let $V=\{1,2,\ldots, N\}$ where $N\ge 7$ and let $\EE$ be the set of all 2-intervals of $V$.
Then the hypergraph $\HH=(V,\EE)$ is $ABABA$-free if and only if the ordering of $V$ is ${1,2,\ldots, N}$ or ${N,N-1,\ldots, 1}$.
\end{lemma}
\begin{proof}
It is easy to see that the orderings ${1,2,\ldots, N}$ and ${N,N-1,\ldots, 1}$ are indeed $ABABA$-free. Let $\pi$ be any linear ordering of $[N]$ other than $1,2,\ldots,N$ and $N,N-1,\ldots,1$. Then there must be an $x\in [N]$ such that $x$ and $x+1$ are not neighbours in $\pi$. Since $N\ge 7$, there is a $y\in[N]$ such that $y$ is not a neighbour of $x$ or $x+1$. As $x, x+1$ and $y$ are not neighbours in $\pi$, we can pick numbers $z$ and $q$ that separate them from each other in the order $\pi$.  Then $\{x,x+1,y\}$ is a 2-interval which forms an $ABABA$ pattern with the 2-interval $\{z,q\}$.
\end{proof}

Let $\mathcal{J}_k$ denote the collection of all intervals on $\{1,2,\ldots,k\}$.
Analogously to Corollary \ref{cor:circularblocks}, Lemma \ref{lem: two intervals} implies the following.

\begin{corollary}\label{cor:2intervals}
Suppose $k\ge 7$ and $V=\cup_{i=1}^kA_i$ for some disjoint $A_i\subset V$. Let $\EE=\{(\bigcup\limits_{i\in\alpha}A_i)\cup(\bigcup\limits_{j\in\beta}A_j):\alpha,\beta\in\mathcal{J}_k\}$.
Then, an ordering $\pi$ of $V$ is $ABABA$-free if and only if $\pi$ has structure $A_1,A_2,\ldots, A_k$ or $A_k,A_{k-1},\ldots, A_1$.
\end{corollary}

The proof of Theorem \ref{thm:ababa} follows the same ideas as the proof of Theorem \ref{thm:abab}.
There might be an easy reduction directly from $ABAB$-freeness, but we are not aware of it and give an explicit proof below.

\begin{proof}
Once again, we will use that finding 2-colorings for 3-uniform hypergraphs is NP-hard. Let $\GG$ be a 3-uniform hypergraph and let $\HH=(V,\EE)$ be constructed the same way as in the proof of Theorem \ref{thm:abab}. As we have seen, $\GG$ is 2-colorable if and only if $\HH$ is $ABAB$-free. We will construct a hypergraph $\HH'$ such that $\HH$ is $ABAB$ free if and only if $\HH'$ is $ABABA$-free. 

Recall that the vertex set $V$ of $\HH$ is $[N]$ and $V=(\bigcup_{i=1}^n R_i)\cup(\bigcup_{j=1}^m S_j)$. We will build $\HH'$ on the vertex set $[N+1]$. Let $\mathcal{I}$ be the collection of intervals on $[N+1]$ that are unions of some of the sets from $R_1,\dots,R_n,S_1,\dots,S_m,\{N+1\}$.
Let $\EE_0$ be the 2-intervals formed from intervals in $\mathcal{I}$, that is $\EE_0=\{I\cup J:\;I,J\in\mathcal{I}\}$. Note that by Corollary \ref{cor:2intervals}, if the edge set of a hypergraph on $[N+1]$ contains $\EE_0$, then any $ABABA$-free ordering will have structure $R_1,\dots,R_n,S_1,\dots,S_m,\{N+1\}$ or $\{N+1\},S_m,\dots,S_1,R_n,\dots,R_1$.

Consider the hypergraph $\HH'=(V',\EE')$, where $V'=[N+1]$ and $\EE'=\EE\cup\EE_+\cup\EE_0$ with $\EE_+=\{E\cup \{N+1\}:\;E\in\EE\}$. We need to show that $\HH$ is $ABAB$-free if and only if $\HH'$ is $ABABA$ free. First, let $\pi$ be any $ABAB$-free ordering of $[N]$ for $\HH$. Recall that for any $ABAB$-free ordering $\pi$ of $\HH$ there is an equivalent $ABAB$-free ordering $\pi_2$ that has structure $R_1,\dots,R_n,S_1,\dots,S_m$.
 We will show that the ordering $\pi'$ that we get by placing $N+1$ after $\pi_2$ is an $ABABA$-free ordering for $\HH'$.

As no two edges from $\EE$ form an $ABAB$ pattern and we introduced only a single new vertex, no two edges from $\EE\cup \EE_+$ can form an $ABABA$ pattern. Hence, if $H$ and $L$ form an $ABABA$ pattern, we can assume that $H\in \EE_0$. Then, as $\pi'$ has structure $R_1,\dots,R_n,S_1,\dots,S_m,\{N+1\}$, the edge $H$ is not only a 2-interval with respect to the usual ordering, but also with respect to $\pi'$. Hence, the only possibility is that $H$ plays the role of $B$ in the pattern. This also implies that $L$ is from $\EE\cup \EE_+$. 

Let $D_1$ be the first and $D_2$ the last of the sets $R_1,\dots, R_n,S_1,\dots,S_m,\{ N+1\}$  that intersects $L$. Since $H$ and $L$ form an $ABABA$ pattern and $H$ is a 2-interval, $H$ must lie between $D_1$ and $D_2$. Recall that $\EE=\EE_1\cup\EE_2\cup\EE_3$ and from Observation \ref{obs:almostinterval} if $L\in \EE_1\cup \EE_2$, then the sets in the list $R_1,\dots, R_n,S_1,\dots,S_m$ between $D_1$ and $D_2$ are subsets of $L$. This would imply that $H\subset L$, a contradiction. So we must have $L\in \EE_3\cup\EE_+$. From Lemma \ref{lem:complement} we know that an edge from $\EE_3$ cannot form an $ABAB$-pattern with anything. So, let $L\in\EE_+$. Since $L$ forms an $ABABA$ pattern with $H$, the edge $L'=L\setminus\{N+1\}\in\EE$ must form an $ABAB$ pattern with $H$; a contradiction by the previous arguments.
Hence $\pi'$ is $ABABA$-free for $H'$.

Conversely, let $\HH'$ be $ABABA$-free and $\pi'$ be any $ABABA$-free ordering for $\HH'$.
By Lemma \ref{lem: two intervals}, $N+1$ should be at the last or the first position in $\pi'$.
Assume the former wlog.
We claim that the ordering $\pi$ that we get by deleting $N+1$ from $\pi'$ is an $ABAB$-free ordering for $\HH$.
Indeed, if there is an $ABAB$ pattern for two hyperedges $H,L$ of $\HH$ in $\pi$, then $H,L$ are also hyperedges of $\HH'$ and they form an $ABAB$ pattern in $\pi'$.
Assume, wlog, that $H$ plays the role of $A$ in the $ABAB$ pattern.
This implies that there is an $ABABA$ pattern formed by $H_+=H\cup\{N+1\}$ and $L$ in $\pi'$ for the hypergraph $\HH'$; a contradiction.
Hence, $\HH$ is $ABAB$-free.
This completes the proof.
\end{proof}
\section{Proof of Theorem \ref{thm:even} and Theorem \ref{thm:odd}}\label{sec:oddeven}
We show Theorem \ref{thm:odd} and Theorem \ref{thm:even} using induction on $k$.  The base cases are covered by Theorem \ref{thm:abab} and Theorem \ref{thm:ababa}.

For a hypergraph $\HH=(V,\EE)$ and $t\in \mathbb{N}$, let $\HH_t$ denote the hypergraph whose vertex set is $V'=V\cup X$ for some new vertices $X=\{x_1,x_2,\ldots,x_{t}\}$ and the edge set is $\EE'=\EE\cup\{E\cup \{x\}:E\in\EE,\; x\in X\}$. 

\begin{lemma}\label{lem:odd-odd}
Let $\HH=(V,\EE)$ be a given hypergraph. Then for any $k\in \mathbb{N}$ the hypergraph $\HH$ is $(AB)^{k-1}A$-free if and only if $\HH_{2k+1}$ is $(AB)^kA$-free.
\end{lemma}

\begin{proof}
Let $\HH$ be $(AB)^{k-1}A$-free and $\pi$ be an $(AB)^{k-1}A$-free ordering of $V$. We will show that $\HH_{2k+1}=(V',\EE')$ is $(AB)^kA$-free.
Let $\pi=v_1,v_2,\dots, v_n$.
Consider the ordering $\pi'=v_1,\dots, v_n,x_1,\ldots,x_{2k+1}$ of $V'$.
If there are hyperedges $H',L'\in\EE'$ which form an $(AB)^kA$ pattern, then there are vertices $h_1,l_1,\ldots,h_k,l_k,h_{k+1}$ in this order in $\pi'$ such that $h_i\in H'\setminus L'$ for $1\le i\le k+1$ and $l_j\in L'\setminus H'$ for $1\le j\le k$.
Let $H,L\in\EE$ be the intersections of $H',L'$ respectively, with the vertex set $V$.
Since $H'$ and $L'$ contain at most one vertex of $X$, we have $|H'\setminus H|\le 1$ and $|L'\setminus L|\le 1$.
Also, as $\HH$ is $(AB)^{k-1}A$-free, in the sequence $h_1,l_1,\ldots,h_k,l_k,h_{k+1}$, at least the last three vertices must be from $X$ otherwise $H$ and $L$ form an $(AB)^{k-1}A$ pattern in the ordering $\pi$.
However, this is not possible because of $|H'\setminus H|\le 1$ and $|L'\setminus L|\le 1$.
Thus $\pi'$ is an $(AB)^{k}A$-free ordering for $\HH_{2k+1}$.

Conversely, assume that $\HH_{2k+1}$ is $(AB)^kA$-free. We will show that $\HH$ is $(AB)^{k-1}A$-free.
Let $\pi'$ be an $(AB)^kA$-free ordering of $V'$ for the hypergraph $\HH_{2k+1}$ and let $\pi$ be an induced ordering of $\pi'$ on the vertex set $V$.
Suppose there are hyperedges $H,L\in\EE$ which form an $(AB)^{k-1}A$ pattern in $\HH$ and let $h_1,l_1,\ldots,h_{k-1},l_{k-1},h_k$ be vertices in this order in $\pi$ such that $h_i\in H\setminus L$ for $1\le i\le k$ and $l_j\in L\setminus H$ for $1\le j\le {k-1}$.
Since, $|X|=2k+1$ and $|\{h_1,l_1,\ldots,h_{k-1},l_{k-1},h_k\}|=2k-1$, it follows that in the ordering $\pi'$,
there must be at least two vertices $x,y\in X$ which lie between two consecutive elements of the sequence $h_1,l_1,\ldots,h_{k-1},l_{k-1},h_k$ or they both occur before or after all the elements of this sequence.
Then either $H\cup \{x\}$ and $L\cup\{y\}$, or $H\cup \{y\}$ and $L\cup\{x\}$ form an $(AB)^kA$ pattern in $\pi'$; a contradiction. Hence, $\pi$ is an $(AB)^{k-1}A$-free ordering for $\HH$.
\end{proof}

\begin{lemma}\label{lem:even-even}
Let $\HH=(V,\EE)$ be a given hypergraph. Then for any $k\in \mathbb{N}$ the hypergraph $\HH$ is $(AB)^{k}$-free if and only if $\HH_{2k+2}$ is $(AB)^{k+1}$-free.
\end{lemma}
\begin{proof}
The proof is the same as that of Lemma \ref{lem:odd-odd}.
\end{proof}

Now the proofs of Theorem \ref{thm:even} and Theorem \ref{thm:odd} are direct implications of the two lemmas above, Theorem \ref{thm:abab} and Theorem \ref{thm:ababa}.

\section{Related questions}\label{sec:questions}

A family of \emph{pseudocircles} is a set of closed Jordan curves such that every two of them are either disjoint, intersect at exactly one point in which they touch, or intersect at exactly two points in which they properly cross each other. A family of \emph{pesudodisks} is a collection of compact planar regions whose boundaries form a family of pseudocircles.

A family $\mathcal{A}$ of pseudo disks in the plane is called a \emph{stabbed} pseudodisk arrangement if there is a point in the plane that lies in every pseudodisk of $\mathcal{A}$.

\begin{theorem}[Ackerman, Keszegh, Pálvölgyi \cite{ackerman2020coloring}]\label{thm:pseudo}
    A hypergraph can be realized as the incidence hypergraph of stabbed pseudodisks and points in the plane if and only if it is $ABAB$-free.
\end{theorem}

\begin{theorem}
It is NP-complete to decide whether an abstract hypergraph can be realized as the incidence hypergraph of points and pseudodisks in the plane or not.
\end{theorem}

\begin{proof}
The problem is in NP, because we can give a short combinatorial description of the realizing arrangement.

To prove NP-hardness, we show that $ABAB$-freeness is reducible to recognizing pseudodisk hypergraphs. Let ${\HH=(V,\EE)}$ be a given hypergraph and we want to decide whether it is $ABAB$-free. Consider the hypergraph $\HH'=(V',\EE')$ where $V'=V\cup \{x\}$ and $\EE'=\{E\cup\{x\}:E\in\EE\}$.
Clearly, $\HH'$ can be represented as a hypergraph of pseudodisks in the plane if and only if $\HH'$ can be represented as a hypergraph of \emph{stabbed} pseudodisks in the plane.
By Theorem \ref{thm:pseudo}, a hypergraph is a hypergraph of stabbed pseudodisks in the plane if and only if it is $ABAB$-free. Furthermore, as $x$ lies in every edge, it cannot take part in an $ABAB$ pattern. Hence $H'$ is $ABAB$-free if and only if $H$ is $ABAB$-free. Therefore, $H$ is $ABAB$-free if and only if  $H'$ can be represented as the incidence hypergraph of points and pseudodisks in the plane.

Applying Theorem \ref{thm:abab} we conclude that it is NP-hard to decide if $\HH$ is a hypergraph of points and pseudodisks in the plane.
\end{proof}

\subsection{Open questions}\label{sec:open}
Here we collect a number of open questions. The most natural is the missing $k=1$ case in Theorem \ref{thm:odd}. 

\begin{problem}
    Is it NP-complete to decide if a given hypergraph is $ABA$-free?
\end{problem}

We mention some related results.
It was proved by Opatrny  \cite{opatrny} that the related BETWEENNESS problem is NP-complete.
Here our input is a collection of ordered triples of some base set $V$, and the question is whether there is an ordering of $V$ in which for all ordered triples the middle element is between the other two elements of the triple in the ordering.
It is not hard to see that this implies that NON-BETWEENNESS is also NP-complete, i.e., when our input is a collection of ordered triples of some base set $V$, and the question is whether there is an ordering of $V$ for which in all ordered triples the middle element is \emph{not} in the middle.
Deciding $ABA$-freeness is a special case of NON-BETWEENNESS.
We also mention a recent result about Helly properties of such relations \cite{gezaek} (and note that $ABA$-free hypergraphs do not have the Helly property, as shown by the $2$-uniform (hyper)graphs $C_n$). 

In our proofs there are some hyperedges that have linear number of vertices.
As mentioned in the introduction, the $ABAB$-freeness of a 2-uniform hypergraph can be decided in polynomial time since they are exactly the outerplanar graphs \cite{ackerman2020coloring}.
However, we do not know if the problem is in $P$ even for 3-uniform hypergraphs; maybe these problems become tractable if we bound the sizes of the hyperedges.

In the proofs we have used $k$-intervals often which is strongly connected to the so-called consecutive ones property. Let us say that a $(0,1)$-matrix has the \emph{$k$-consecutive ones property} if there exists a row order such that in each column the occurrences of all ones appear in at most $k$ consecutive blocks. Clearly, the incidence matrix of a hypergraph has the $k$-consecutive ones property if and only if there is an ordering of the vertices such that every hyperedge is a union of at most $k$ intervals in that ordering.

Deciding if a given (0,1)-matrix has the  $1$-consecutive ones property can be done in polynomial time \cite{BOOTH1976335,pjm/1102995572}. On the other hand, for $k\ge 2$ it is NP-complete to decide whether a given (0,1)-matrix has the $k$-consecutive ones property \cite{goldberg1995four}. This latter result would also follow from our methods with simple modifications, even for the circular k-interval case, but we omit these.

\acknowledgements
\label{sec:ack}
This research started during the special semester on Discrete Geometry and Convexity at the Erdős Center, Budapest, in 2023, supported by ERC Advanced Grants ``GeoScape" and ``ERMiD", no. 882971 and 101054936.
The research was also partially supported by IIIT-Delhi through the Overseas Research Fellowship.

\nocite{*}
% \bibliography{ref}
\bibliographystyle{abbrvnat}
% use the following instead if you encounter problems 
%\bibliographystyle{alpha}
\bibliography{ref}
\label{sec:biblio}

\end{document}